\documentclass[a4paper]{article}

\usepackage{amsmath, amssymb, amsthm}%, wasysym}
\usepackage[latin1]{inputenc}
\usepackage[english]{babel}
\usepackage[T1]{fontenc}

\newtheorem{thm}{Theorem}

\theoremstyle{remark}
\newtheorem{rem}[thm]{Remark}
\theoremstyle{definition}

\def\t{\theta}
\def\vt{\vert\theta\vert}
\def\b{\beta}
\def\e{\varepsilon}
\newcommand{\R}{\mathbb{R}}
\newcommand{\N}{\mathbb{N}}
\newcommand{\ud}{\mathrm{d}}

\newcommand{\be}{\begin{equation}}
\newcommand{\ee}{\end{equation}}
\newcommand{\fer}[1]{\eqref{#1}}
\newcommand{\bee}{\begin{equation*}}
\newcommand{\eee}{\end{equation*}}
\newcommand{\coloneq}{\mathrel{\mathop :}=}
\title{The grazing collision limit of Kac \\ caricature of Bose-Einstein particles}
\author{Thibaut Allemand\footnote{DMA, \'Ecole Normale Sup\'erieure, 45 rue d'Ulm, 75230 Paris Cedex 05, France}
 \and Giuseppe Toscani\footnote{Dipartimento di matematica, Universit\`a di Pavia, Via Ferrata 1, 27100 Pavia, Italy}}
\date{\today}
\begin{document}

\maketitle

\begin{abstract}
We  discuss the grazing collision limit of certain kinetic models of
Bose-Einstein particles obtained from a suitable modification of the
one-dimensional Kac caricature of a Maxwellian gas without cut-off.
We recover in the limit a nonlinear Fokker-Planck equation which
presents many similarities with the one introduced by Kaniadakis and
Quarati in \cite{KQ94}. In order to do so, we perform a study of the
moments of the solution. Moreover, as is typical
in Maxwell models, we make an essential use of the Fourier version
of the equation.

\end{abstract}

\tableofcontents

\section{Introduction}

The quantum dynamics of many body systems is often modelled by a nonlinear
Boltzmann equation which exhibits a gas-particle-like collision behavior. The
application of quantum assumptions to molecular encounters leads to some
divergences from the classical kinetic theory \cite{chap-cow} and despite their
formal analogies the Boltzmann equation for classical and quantum kinetic
theory present very different features. The interest in the quantum framework
of the Boltzmann equation has increased noticeably in the recent years.
Although the quantum Boltzmann equation, for a single specie of particles, is
valid for a gas of fermions as well as for a gas of bosons, blow up of the
solution in finite time can occur only in the bosonic case at low temperature. As a consequence the
quantum Boltzmann equation for a gas of bosons represents the most challenging
case both mathematically and numerically. In particular this equation has been
successfully used for computing non-equilibrium situations where Bose-Einstein
condensate occurs.  From Chapman and Cowling \cite{chap-cow} one can learn that
the Boltzmann Bose-Einstein equation (BBE) is established by imposing that,
when the mean distance between neighboring molecules is comparable with the
size of the quantum wave fields with which molecules are surrounded, a state of
congestion results. For a gas composed of Bose-Einstein identical particles,
according to quantum theory, the presence of a like particle in the
velocity-range $dv$ increases the probability that a particle will enter that
range; the presence of $f(v) dv$ particles per unit volume increases this
probability in the ratio $1+\delta f(v)$. This fundamental assumption yields
the Boltzmann Bose-Einstein equation
 \begin{equation}
\label{BBE1} \displaystyle{\frac{\partial f}{\partial t}} =
Q_{QBE}(f)(t,v),\qquad t\in \R_+,~~v\in\R^3,
 \end{equation}
where
 \begin{equation}
\label{ker}
\begin{split}
Q_{QBE}(f)(t,v) = \int_{\R^3 \times S^2} B(v-v_*, \omega)&\big( f'f_*'(1+\delta f)(1+\delta f_*)\\
& - ff_*(1+\delta f')(1+\delta f_*')\big)\ud v_* d\omega,
\end{split}
 \end{equation}
where as usual we denoted
 \[
 f= f(v), \quad f_* = f(v_*), \quad f' = f(v'), \quad f'_* = f(v'_*),
 \]
and the pairs $(v,v_*)$ (respectively $(v',v'_*)$) are the post- (respectively
pre-) collision velocities in a elastic binary collision. In \fer{ker}
$B(z,\omega)$ is the collision kernel which is a nonnegative Borel function of
$|z|, |<z,\omega>|$ only
 \be\label{cs}
 B(z,\omega) = B \left(|z|,\frac {<z,\omega>}{|z|}\right),  \qquad (z,\omega)
 \in \R^3 \times S^2.
 \ee
The solutions $f(v,t)$ are velocity distribution functions (i.e.,
the density functions of particle number), $\delta =(h/m)^3/g$, $h$
is the Planck's constant, $m$ and $g$ are the mass and the
``statistical weight'' of a particle (see \cite{Lu} for details).

For a non relativistic particle, by setting $v(p) = p/m$,  the
 collision operator $Q_{QBE}$ can be rewritten in general form as follows
 \cite{ST95, ST97}
\begin{equation}
\label{ker-p}
\begin{split}
Q_{QBE}(f)(t,p) = \int_{\R^9} W(p, p_*, p', p'_*) &\big( f'f_*'(1+\delta f)(1+\delta f_*)\\
& - ff_*(1+\delta f')(1+\delta f_*')\big)dp_*dp' dp'_*
\end{split}
 \end{equation}
where $W$ is a nonnegative measure called transition rate, which is
of the form
 \[
W(p, p_*, p', p'_*) = \Omega(p, p_*, p', p'_*)\delta (p,+ p_*- p'- p'_*)
\delta( \mathcal{E}(p) +  \mathcal{E}(p_*)-  \mathcal{E}(p')-
\mathcal{E}(p'_*)),
 \]
 where $\delta$ represents the Dirac measure and $\mathcal{E}(p)$ is the
 energy of the particle. The quantity $Wdp' dp'_*$ is the probability for the
 initial state $(p,p_*)$ to scatter and become a final state of two particles
whose momenta lie in a small region $dp' dp'_*$. The function $\Omega$ is
directly related to the differential cross section (see \fer{cs}), a quantity
that is intrinsic to the colliding particles and the kind of interaction
between them.
 The collision operator \fer{ker-p} is simplified by assuming a boson distribution
which only depends on the total energy $e = \mathcal{E}(p)$. In this
last case $f=f(e,t)$ is the boson density in energy space.

Together with the Boltzmann description given by the collision operators
\fer{ker}-\fer{ker-p}, other kinetic models for Bose-Einstein particles have
been introduced so far. In particular, a related model described by means of
Fokker-Planck type non linear operators has been proposed by Kompaneets
\cite{kompa} to describe the evolution of the radiation distribution $f(x,t)$
in a homogeneous plasma when radiation interacts with matter via Compton
scattering
 \begin{equation}
\label{kompa} Q_{K}(f)(t,\rho) = \frac 1{\rho^2}\frac{\partial}{\partial
\rho}\left[ \rho^4\left( \frac{\partial f}{\partial \rho} + f + f^2 \right)
\right], \qquad \rho\in (0, +\infty)
 \end{equation}
In that context the coordinate $\rho$ represents a momentum coordinate, $\rho =
|p|$. More precisely, an equation which includes \fer{kompa} as a particular
case is obtained in \cite{kompa} as a leading term for the corresponding
Boltzmann equation under the crucial assumption that the scattering cross
section is of the classical Thomson type (see \cite{EMV03} for details).

The fundamental assumption which leads to the correction in the
Boltzmann collision operator \fer{ker}, namely the fact that the
presence of $f(v) dv$ particles per unit volume increases the
probability that a particle will enter the velocity range $dv$ in
the ratio $1+\delta f(v)$, has been recently used  by Kaniadakis and
Quarati \cite{KQ94, KQ93} to  propose a correction to the drift term
of the Fokker-Planck equation in presence of quantum
indistinguishable particles, bosons or fermions. In their model, the
collision operator \fer{ker} is substituted by
 \begin{equation}
\label{kerFP}
 Q_{FP}(f)(t,v) = \nabla\cdot \left[ \nabla f + vf(1+\delta
f)\right].
 \end{equation}
Maybe the most remarkable difference between the kinetic operators \fer{ker} and \fer{kerFP} is that, while the former is such that mass, momentum
and energy are collision invariant, the latter does not admit the energy as collision invariant. This suggests that the operator \fer{kerFP} would
not result directly through an asymptotic procedure from the Bose-Einstein collision operator \fer{ker}, but instead from some linearized version,
where only the mass is preserved under the collision mechanism.

For a mathematical analysis of the quantum Boltzmann equation in the space homogeneous isotropic case we refer to \cite{EM99,EM01, EMV03, Lu,
Lu1}. We remark that already the issue of giving mathematical sense to the collision operator is highly nontrivial (particularly if positive
measure solutions are allowed, as required by a careful analysis of the equilibrium states). All the mathematical results, however, require very
strong cut-off assumptions on the cross-section \cite{EMV03, Lu}.

Also,  accurate numerical discretization of the quantum Boltzmann equation, which maintain the basic analytical and physical features of the
continuous problem, namely, mass and energy conservation, entropy growth and equilibrium distributions have been introduced recently in
\cite{BMP04, MP05}. Related  works \cite{Le98, PRT00, PTV03} in which fast methods for Boltzmann equations were derived using different techniques
like multipole methods, multigrid methods and spectral methods, are relevant to quote.

At the Fokker-Planck level, the qualitative analysis of the Kompaneets equation described by the operator \fer{kompa} has been exhaustively
studied in \cite{EHV98}, while the numerical simulation has been done by Buet and Cordier \cite{BC02}. To our knowledge, the mathematical study of
the Fokker-Planck equation \fer{kerFP} introduced by Kaniadakis and Quarati \cite{KQ94} has been done only very recently \cite{CRS06}, where the
one-dimensional version of \fer{kerFP} has been studied.

In the case of the quantum Boltzmann equation the asymptotic equivalence between the binary collision operators \fer{ker}, \fer{ker-p} and the
Fokker-Planck type operators \fer{kompa} and \fer{kerFP}  is unknown. This is not the case for the classical binary collisions in a elastic gas,
where the asymptotic equivalence between the Boltzmann and the Fokker-Planck-Landau equations has been proven rigorously in a series of papers by
Villani \cite{Vi98, Vi02} by means of the so-called \emph{grazing collision} asymptotics.

The same asymptotic procedure, in the case of the one-dimensional Kac equation \cite{Kac}, showed the asymptotic equivalence between Kac collision
operator and the linear Fokker-Planck operator \cite{to98}. The method of proof in \cite{to98} takes advantage from the relatively simple
structure of Kac equation. Taking this into account, in order to establish the asymptotic connection between the Boltzmann equation for Bose
Einstein particles and its Fokker-Planck description, we will introduce a one-dimensional kinetic model in the spirit of Kac caricature of a
Maxwell gas with a singular kernel. Then we will study the grazing collision limit of the equation, which leads to a Fokker-Planck type equation
in which the drift is of the form of equation \fer{kerFP}, but the coefficient of the (linear) diffusion term depends on time through the density
function. More precisely, the Fokker-Planck collision operator reads
 \begin{equation}
\label{kerFP2}
 Q_{FP}(f)(t,v) = A_t(f) \frac{\partial^2 f}{\partial v^2}+
   B_t(f)\frac{\partial}{\partial v}(v f (1+\delta f)),
 \end{equation}
where
 \bee
 A_t(f) =  \int_\R v^2 f(v, t)(1+\delta f(v,t))\ud v
 \eee
and
\bee
B_t(f) = \int_\R f(v,t) \ud v.
\eee
The paper is organized as follows. In the next Section we will introduce the model, together with some simplifications (mollified model), then in Section \ref{formal} we formally show the convergence in the grazing collision limit. We then focus on a mollified model to get rigorous results: in Section \ref{existence} we prove existence of a weak solution,  in Section \ref{sectionmoments} we focus on the moments
of the solution and on some regularity properties. Finally, in Section \ref{grazing}, we will deal with the grazing collision limit.

\section{The Kac caricature of a Bose Einstein gas}\label{k}

The simplest one-dimensional model which maintains almost all physical properties of the Boltzmann equation for a Bose-Einstein gas can be
obtained by generalizing Kac caricature of a Maxwell gas to Bose-Einstein particles. This one-dimensional model reads as follows
 \begin{equation}
\label{BBE}
\begin{cases}
\displaystyle{\frac{\partial f}{\partial t}} = Q_{QBE}(f)(t,v),\qquad t\in \R_+,~~v\in\R, \\
f(0,v) = f_0(v),
\end{cases}
\end{equation}
where
\begin{equation}
\label{kernel}
\begin{split}
Q_{QBE}(f)(t,v) = \int_{-\frac{\pi}{2}}^{\frac{\pi}{2}} \beta(\theta)\int_{\R}&\big( f'f_*'(1+\delta f)(1+\delta f_*)\\
& - ff_*(1+\delta f')(1+\delta f_*')\big)\ud v_* \ud\theta .
\end{split}
\end{equation}
For the sake of brevity, we used the usual notations
$$
f \equiv f(t,v),\quad f'\equiv f(t,v'),\quad f_* \equiv f(t,v_*),\quad f_*'\equiv f(t,v_*'),
$$
The initial datum $f_0$ is a nonnegative measurable function. The pre-collision velocities $(v',v_*')$ are defined by the Kac rotation rule
\cite{Kac}, which is given by
\begin{equation}
\label{collision}
\begin{cases}
v' = v\cos\theta - v_* \sin\theta \\
v_*' = v\sin\theta + v_*\cos\theta .
\end{cases}
\end{equation}
Collisions \fer{collision} imply the conservation of the energy at each collision
 \be \label{micrenergy} v^2+v_*^2 = v'^2 + v_*'^2.
 \ee
 Let us observe that the system \eqref{collision} can be reversed so that we can write the post-collision velocities with respect to the
pre-collision ones
\begin{equation}
\label{precollision}
\begin{cases}
v = v'\cos\theta + v_*' \sin\theta \\
v_* = -v'\sin\theta + v_*'\cos\theta.
\end{cases}
\end{equation}
The parameter $\delta$ in \eqref{kernel} is a positive constant. The choice $\delta=0$ would lead us back to standard  Kac model, whereas $\delta$
negative would lead us to the Boltzmann-Fermi-Dirac equation, whose features are very different from that of the Boltzmann-Bose-Einstein equation.

The cross-section $\beta(\theta)$ is a function defined over $(-\frac{\pi}{2},\frac{\pi}{2})$. In the original Kac equation \cite{Kac},
$\beta(\theta)$ is assumed constant, which implies that collisions spread out uniformly with respect to the angle $\theta$. Following Desvillettes
\cite{desville2}, we will here assume that the cross-section is suitable to concentrate collisions on the \emph{grazing} ones (these collisions
are those that are neglected when the cut-off assumption is made). This corresponds to satisfy one or more of the following hypotheses

\begin{description}
\item[H1] $\beta(\theta)$ is a nonnegative even function.
\item[H2] $\beta(\theta)$ satisfies a non-cutoff assumption on the form
\be \label{crosssection} \beta(\theta) \sim \frac{1}{|\theta|^{1+\nu}} \qquad
\textrm{when}~~\theta \to 0, \ee with $1<\nu<2$. That is,
$$
\int_{-\frac{\pi}{2}}^{\frac{\pi}{2}}\beta(\theta)|\sin \theta|\ud\theta =
+\infty
$$
whereas

$$
\int_{-\frac{\pi}{2}}^{\frac{\pi}{2}}\beta(\theta)|\sin
\theta|^{\nu+\varepsilon}\ud\theta < +\infty
$$
for all $\varepsilon >0$.
\item[H3] $\beta(\theta)$ is zero near $-\frac{\pi}{2}$
and $\frac{\pi}{2}$, namely there exists a small constant $\varepsilon_0>0$ such that
$$
\beta(\theta) = 0 \qquad \forall \theta \in
(-\frac{\pi}{2},-\frac{\pi}{2}+\varepsilon_0] \cap
[\frac{\pi}{2}-\varepsilon_0,\frac{\pi}{2}).
$$

\end{description}

In the case in which the classical Kac equation is concerned, the asymptotic equivalence between the non cut-off Kac equation and the linear
Fokker-Planck equation as collisions become grazing has been proven in \cite{to98}. Hence, the passage to grazing collisions in \fer{kernel},
would give us the correct Fokker-Planck type operator which leads the initial density towards the Bose-Einstein distribution.

Due to the symmetries of the kernel \eqref{kernel} and to the microscopic conservation of the energy \eqref{micrenergy}, it can be easily shown,
at least at a formal level, that the mass and the global energy of the solution are conserved
$$
\int_{\R} f(t,v)\ud v = \int_{\R}f_0(v)\ud v ,
$$
and
$$
\int_{\R} v^2f(t,v)\ud v = \int_{\R}v^2f_0(v)\ud v ,
$$
for all $t>0$. Moreover, if
$$
H(f) = \int_{\R} \left( \frac{1}{\delta}(1+\delta f)\log (1+\delta f) -f\log f
\right)\ud v
$$
denotes the Bose-Einstein entropy, the time derivative of $H(f)$ is given by
\bee
\begin{split}
D(f) = \frac{1}{4}\int_{-\frac{\pi}{2}}^{\frac{\pi}{2}}\beta(\theta)\int_{\R^2}\Gamma\Big( ff_*'(1+\delta f)(1+&\delta f_*),\\
& ff_*(1+\delta f')(1+\delta f_*')\Big)\ud v\ud v_*\ud\theta
\end{split}
\eee with

\be \Gamma(a,b) = \begin{cases}
        (a-b)\log (a/b),\quad &a>0,~~b>0; \\
        +\infty,        &a>0,~~b=0~~\textrm{or} ~~a=0,~~b>0;\\
        0,          &a=b=0.
              \end{cases}
\ee Then, since $D(f) \ge 0$ the solution $f(t)$ to equation
\fer{BBE} satisfies formally an $H$-theorem: $H(f(t))$ is
monotonically increasing unless $f(t)$ coincides with the
Bose-Einstein density $f_{BE}$, defined by the relationship
 \be\label{bose-e}
 \frac{f_{BE}(v)}{1 + \delta f_{BE}(v)} = a e^{-bv^2},
 \ee
where $a$ and $b$ are positive constant chosen to satisfy the mass and energy conservation for $f_{BE}$.

It can be easily verified by direct inspection that the fourth order term in
\fer{kernel} cancels out from the collision integral, so that it can be
rewritten as
\begin{equation}
\label{kerner2}
\begin{split}
Q_{QBE}(f)(t,v) = \int_{-\frac{\pi}{2}}^{\frac{\pi}{2}} \beta(\theta)\int_{\R}&\big( f'f_*'(1+\delta f+\delta f_*)\\
& - ff_*(1+\delta f'+\delta f_*')\big)\ud v_* \ud\theta.
\end{split}
\end{equation}

In trying to give a rigorous signification to equation \eqref{BBE}, several difficulties arise. In fact, our non-cutoff cross-section
$\beta(\theta)$ does not allow us neither to use the same change of variable as in \cite{Lu}, nor to use the same weak formulation as in
\cite{desville2}. A sufficient condition to give a sense to the collision kernel would be that $f\in L^\infty(\R_+\times\R)$. It would even be
enough that such a condition hold for the quantum part, that is for the $f$ involved in the terms of the form $1+\delta f$.

To satisfy that condition, we introduce a new model where the  quantum part is smoothed. Let $\psi$ be a mollifier, that is

\begin{enumerate}
  \item $\psi \in C_c^\infty(\R)$
  \item $\psi \geq 0$
  \item $\int_\R \psi(v)\ud v =1$.
\end{enumerate}
Then, let
$$
\tilde f(t,v) = \int_\R f(t,v-w)\psi(w)\ud w = f(t,.)*_v \psi.
$$
The function $\tilde f$ is regular in the velocity variable, and relies uniformly in all the $L^p(\R)$ spaces (for $1\leq p \leq +\infty$) since
the $L^1(\R)$ norm of $f(t,.)$ is constant. Moreover, $\tilde f(t,.)$ is as close as we want to $f(t,.)$ in all these norms, provided $\psi$ is
well chosen, so that our new model is nothing but an approximation of \eqref{BBE}:
\begin{equation}
  \label{BBEmod}
\begin{cases}
  & \displaystyle{\frac{\partial f}{\partial t} = \tilde Q_{QBE}(f)}, \qquad t\in \R_+, v\in\R \\
  & f(0,v) = f_0(v)
\end{cases}
\end{equation}
with

\begin{equation}
  \label{Kernelmod}
\begin{split}
\tilde Q_{QBE}(f)(t,v) = \int_{-\frac{\pi}{2}}^{\frac{\pi}{2}} \beta(\theta)\int_{\R}&\big( f'f_*'(1+\delta \tilde f)(1+\delta \tilde f_*)\\ & -
ff_*(1+\delta  \tilde f')(1+\delta\tilde f_*')\big)\ud v_* \ud\theta.
\end{split}
\end{equation}
This approximation still formally preserves mass and energy, while maintaining the same nonlinearity of the original collision operator. It has to
be remarked, however, that both the validity of the $H$-theorem and the explicit form of the steady solution are lost. Other approximations can be
introduced, which do not exhibit this problem. Among others, the operator
 \begin{equation}
\label{good}
\begin{split}
\tilde Q_{QBE}(f)(t,v) = \int_{-\frac{\pi}{2}}^{\frac{\pi}{2}} \beta(\theta)\int_{\R}&\big( \frac{f'}{1+\delta f'}\frac{f_*'}{1+\delta f_*'}
 - \frac{f}{1+\delta f}\frac{f_*}{1+\delta f_*}\big)\cdot \\ & \cdot(1+\delta \tilde f)(1+\delta \tilde f_*)(1+\delta \tilde f')
 (1+\delta \tilde f_*')\ud v_*
\ud\theta.
\end{split}
\end{equation}
preserves mass and energy, satisfies the $H$-theorem and possesses the right steady state. Unlikely, the nonlinearity of \fer{good} is difficult
to handle for our purposes since its Fourier transform can not be written in a simple form.

%The remaining of our work will concentrate on the study of \eqref{BBEmod}. In Section \ref{existence}, we will study the existence of a solution
%to this problem. Then in Section \ref{sectionmoments}, we will focus on the moments of the solution and on some regularity. In Section
%\ref{grazing}, we will finally study the grazing collision limit. Our work will end in Section \ref{formal} with some partial result for the non
%mollified model \eqref{BBE}.

Let us end this Section with a few notations. The functional spaces that will be used in the following, apart from the usual Lebesgue spaces, are
the weighted Lebesgue spaces, defined, for $p>0$, by the norm
$$
\|g\|_{L^1_p(\R)} = \int_\R (1+|v|^p)|g(v)|\ud v.
$$
We will also need some Sobolev spaces, defined for $0<s<1$ by the norm
$$
\|g\|_{H^s(\R)}^2= \|g\|_{L^2}^2 + |g|^2_{H^s}
$$
where
$$
|g|^2_{H^s} = \iint_{\R^2}\frac{|g(x+y)-g(y)|^2}{|y|^{1+2s}}\ud x\ud y.
$$
Our convention for the Fourier transform is the following:
$$
\hat f (\xi)=\mathcal F (f)(\xi) = \int_\R f(v)e^{-iv\xi}\ud v
$$
and the inverse Fourier transform is given by
$$
f(v) = \frac{1}{2\pi}\int_\R \hat f (\xi)e^{iv\xi}\ud\xi.
$$
We will sometimes use the notations
$$
m = \int_\R f_0(v)\ud v
$$
and
$$
e = \int_\R v^2 f_0(v)\ud v.
$$

\section{Formal results}
\label{formal}

In this section we will show how the grazing collision limit work on the Kac model for bosons \eqref{BBE2}

 \begin{equation}
\label{BBE2}
\begin{cases}
\displaystyle{\frac{\partial f}{\partial t}} = Q_{QBE}(f)(t,v),\qquad t\in \R_+,~~v\in\R, \\
f(0,v) = f_0(v),
\end{cases}
\end{equation}
where

\begin{equation}
\label{kernel2}
\begin{split}
Q_{QBE}(f)(t,v) = \int_{-\frac{\pi}{2}}^{\frac{\pi}{2}} \beta(\theta)\int_{\R}&\big( f'f_*'(1+\delta f)(1+\delta f_*)\\
& - ff_*(1+\delta f')(1+\delta f_*')\big)\ud v_* \ud\theta.
\end{split}
\end{equation}
As we pointed out before, it is not known how to prove that solutions to this model exist; hence we will give only formal results, assuming that a solution to this equation exists and is regular enough.

\subsection{$H$-theorem and regularity of the solution}

Formally, the solutions to \eqref{BBE2} satisfy the $H$-theorem as pointed out in the previous section:

\begin{thm}
Let $f$ be a solution of the problem \eqref{BBE2}, with $f_0 \in L\log L (\R)$, and assume that $H(f)$ and $D(f)$ are well defined. Then
$$
H(f(t,.)) = H(f_0) + \int_0^t D(f(s,.))\ud s\qquad \forall t>0.
$$
Consequently, the entropy $H$ is increasing along the solution.
\end{thm}

Using the $H$-theorem, we can give an {\it a priori} regularity estimate on the
solution $f$ to \eqref{BBE2}.

\begin{thm}
\label{regularity} Let $\beta$ satisfy the properties (H1), (H2) and
(H3), and let $f(t,v)$ be a solution of the problem \eqref{BBE2}
with the initial datum $f_0 \in L^1_2 \cap L\log L$. Assume that $f$
satisfies the $H$-theorem. Then we have
$$
\log(1+\delta f) \in L^2_{loc}(\R_+;H^{\nu/2}(\R)).
$$
If in addition $f\in L^\infty(\R_+\times\R)$, then

$$
f \in L^2_{loc}(\R_+;H^{\nu/2}(\R)).
$$
\end{thm}

\begin{proof}
We will do the computations as if the cross-section and the function $f$ were smooth. Using the classical changes of variable
$(v,v_*,\theta)\mapsto (v',v_*',-\theta)$ and $(v,v_*)\mapsto(v_*,v)$ which have unit jacobian, we have:

\begin{equation*}
\begin{split}
D(f) &= \frac{1}{4}\int_{-\frac{\pi}{2}}^{\frac{\pi}{2}}\beta(\theta)
\int_{\R^2}\Big(f'f_*'(1+\delta f)(1+\delta f_*) - ff_*(1+\delta f')(1+\delta f_*')\Big) \\
&\qquad\qquad\qquad\qquad\times\log \frac{f'f_*'(1+\delta f)(1+\delta f_*)}
{ff_*(1+\delta f')(1+\delta f_*')}\ud v\ud v_*\ud\theta\\
&= -\frac{1}{2}\int_{-\frac{\pi}{2}}^{\frac{\pi}{2}}\beta(\theta)
\int_{\R^2}\Big(f'f_*'(1+\delta f)(1+\delta f_*) - ff_*(1+\delta f')(1+\delta f_*')\Big)\\
&\qquad\qquad\qquad\qquad\times\log\Big(ff_*(1+\delta f')(1+\delta f_*')\Big)\ud v\ud v_*\ud\theta\\
&=-\int_{-\frac{\pi}{2}}^{\frac{\pi}{2}}\beta(\theta)
\int_{\R^2}\Big(f'f_*'(1+\delta f)(1+\delta f_*) - ff_*(1+\delta f')(1+\delta f_*')\Big)\\
&\qquad\qquad\qquad\qquad\times\log\Big(f(1+\delta f')\Big)\ud v\ud v_*\ud\theta\\
&=\int_{-\frac{\pi}{2}}^{\frac{\pi}{2}}\beta(\theta)
\int_{\R^2}ff_*(1+\delta f')(1+\delta f_*')\log\frac{f(1+\delta f')}{f'(1+\delta f)}\ud v\ud v_*\ud\theta\\
&=\int_{-\frac{\pi}{2}}^{\frac{\pi}{2}}\beta(\theta)
\int_{\R^2}f_*(1+\delta f_*')\Big(f(1+\delta f')\log\frac{f(1+\delta f')}{f'(1+\delta f)} \\
&\qquad\qquad\qquad\qquad\qquad\qquad\qquad-f(1+\delta f')
+ f'(1+\delta f) \Big)\ud v\ud v_*\ud\theta\\
&\qquad\qquad\quad +\int_{-\frac{\pi}{2}}^{\frac{\pi}{2}}
\beta(\theta)\int_{\R^2}f_*(1+\delta f_*')\Big(f(1+\delta f') - f'(1+\delta f) \Big)\ud v\ud v_*\ud\theta\\
&= I_1 + I_2.
\end{split}
\end{equation*}
The term $I_2$ can be treated easily, because it can be written as
$$
I_2 = \int_{-\frac{\pi}{2}}^{\frac{\pi}{2}}\beta(\theta)\int_{\R^2}f_*(1+\delta f_*')(f - f' )\ud v\ud v_*\ud\theta,
$$
and the presence of $f-f'$ involves strong cancellations. In fact, the term $I_2$ verifies
 \be\label{pp}
I_2= \|f\|_{L^1}^2\int_{-\frac{\pi}{2}}^{\frac{\pi}{2}}\beta(\theta)\left(1-\frac{1}{\cos\theta}\right)\ud\theta.
 \ee
To prove \fer{pp}, consider that
 \bee
\begin{split}
I_2 &= \int_{-\frac{\pi}{2}}^{\frac{\pi}{2}}\beta(\theta)
\int_{\R^2}f_*(f - f' )\ud v\ud v_*\ud\theta + \delta\int_{-\frac{\pi}{2}}^{\frac{\pi}{2}}
\beta(\theta)\int_{\R^2}f_*f_*'(f - f' )\ud v\ud v_*\ud\theta \\
&= I_2^1 + I_2^2
\end{split}
\eee
Thanks to the change of variable $(v,v_*,\theta)\mapsto (v',v_*',-\theta)$, we see that $I_2^2=0$. On the second part of $I_2^1$, we use the
change of variable $v\mapsto v'$ with $v_*$ and $\theta$ fixed, which jacobian is
$$
\frac{\ud v}{\ud v'} = \frac{1}{\cos\theta}.
$$
Therefore
$$
I_2^1 = \int_{-\frac{\pi}{2}}^{\frac{\pi}{2}}\beta(\theta)\left(1-\frac{1}{\cos\theta}\right)\int_{\R^2}f_*f\ud v\ud v_*\ud\theta
$$
and from this \fer{pp} follows.

Using now the inequality
$$
x\log\frac{x}{y}-x+y \geq \left( \sqrt{x}-\sqrt{y}\right)^2,\qquad \forall x,y >0,
$$
we obtain
 \bee I_1 \geq \int_{-\frac{\pi}{2}}^{\frac{\pi}{2}}\beta(\theta) \int_{\R^2}f_*(1+\delta f_*')\Big(\sqrt{f(1+\delta
f')}-\sqrt{f'(1+\delta f)}\Big)^2 \ud v\ud v_* \ud\theta.
 \eee
 Therefore
  \be \label{eqreg} \int_{-\frac{\pi}{2}}^{\frac{\pi}{2}}\beta(\theta)
\int_{\R^2}f_*(1+\delta f_*')\Big(\sqrt{f(1+\delta f')}-\sqrt{f'(1+\delta f)}\Big)^2 \ud v\ud v_* \ud\theta \leq D(f) + c_1\|f\|_{L^1}^2 \ee where
$$
c_1 = \int_{-\frac{\pi}{2}}^{\frac{\pi}{2}}\beta(\theta)\left(\frac{1}{\cos\theta}-1\right)\ud\theta>0.
$$
Now, we write
$$
\sqrt{f(1+\delta f')}-\sqrt{f'(1+\delta f)} = (\log (1+\delta f') -
(\log 1+\delta f))\frac{\sqrt{f(1+\delta f')}-\sqrt{f'(1+\delta f)}}{\log (1+\delta f') - \log (1+\delta f)}.
$$
For $0<a<x$, let
$$
\phi(x) = \frac{\sqrt{x(1+\delta a)}-\sqrt{a(1+\delta x)}}{\log(1+\delta a) - \log(1+\delta x)}.
$$
Then, there exists some constant $c_2>0$ that does not depend on $x$ or $a$ such that
$$
|\phi(x)|>c_2 \qquad \forall x>a.
$$
From this inequality we deduce that
$$
c_2\int_{-\frac{\pi}{2}}^{\frac{\pi}{2}}\beta(\theta)
\int_{\R^2}f_*(1+\delta f_*')\Big(\log (1+\delta f') - \log( 1+\delta f)\Big)^2 \leq D(f) + c_1\|f\|_{L^1}^2
$$
It has been shown in \cite{regularitykac} that if $F$ is a real function such that $F(f) \in L^2(\R)$ satisfying
$$
\int_{-\frac{\pi}{2}}^{\frac{\pi}{2}}\beta(\theta)\int_{\R^2}f_*\Big(F ( f') - F( f)\Big)^2 \leq D(f) + c_1\|f\|_{L^1}^2,
$$
then the following inequality holds:
$$
\|F(f)\|_{H^{\nu/2}}^2 \leq D(f) + c_1\|f\|_{L^1}^2.
$$
Taking $F(f) = c_2 \log (1+\delta f)$ the result follows.
\end{proof}

\subsection{Moments of the solution and the grazing collision limit}

We can now make precise assumptions on the
asymptotics of the grazing collisions, namely in letting the kernel $\beta$ concentrate on the singularity $\theta = 0$. We will introduce a
family of kernels $\{ \b_\e(\vt) \}_{\e > 0}$ satisfying hypotheses (H1) and (H2), with
\be
\label{Nbetaepsilon2}
\forall \theta_0>0~~~\sup_{\theta>\theta_0}\b_\e(\vt) \underset{\e\to 0}{\longrightarrow}0
\ee
and
 \be\label{Nbetaepsilon}
\lim_{\e \to 0^+} \int_0^\pi \b_\e(\vt)\t^2 \, d\t = 1
 \ee
This can be obtained in several ways, for example taking, for $0<\mu <1$
 \[
\b_\e(\vt) = \frac{2(1-\mu)}{\e\vt^{2+\mu}} \quad  0 \le |\t| \le
\e^{1/(1-\mu)},
 \]
 \[
 \b_\e(\vt) = \frac{(1-\mu)\e}{\vt^{2+\mu}} \quad {\rm elsewhere}.
 \]

Let $f_\e$ be a solution of
 \begin{equation}
\label{BBE2e}
\begin{cases}
\displaystyle{\frac{\partial f_\e}{\partial t}} = Q_{QBE}^\e(f_\e)(t,v),\qquad t\in \R_+,~~v\in\R, \\
f_\e(0,v) = f_0(v),
\end{cases}
\end{equation}
 where $\beta_\e(\theta)$ has replaced $\beta(\theta)$. The grazing collision limit is obtained when $\e\to 0$. Our strategy is first to take the Fourier transform in the velocity variable of equation \eqref{BBE2e}, then to pass to the limit $\e\to 0$ in the Fourier formulation, and finally to recognize at the limit the Fourier formulation of a quantum Fokker-Planck equation.

Let us take the Fourier transform of \eqref{BBE2e}:

\bee
\begin{split}
\frac{\partial \hat f_\e(t,\xi)}{\partial t} = \int_{-\frac{\pi}{2}}^{\frac{\pi}{2}}\int_{\R^2}\beta_\e(\theta)f_\e f_{\e,*}
(1+\delta  {f_\e}(v')+\delta  {f_\e}(v_*'))(e^{-iv'\xi}-e^{-iv\xi})\ud v\ud v_*\ud\theta.
\end{split}
\eee
We can split the integral on the right-hand side into three parts. The first part gives
 \bee
\begin{split}
&\int_{-\frac{\pi}{2}}^{\frac{\pi}{2}}\int_{\R^2}\beta_\e(\theta)f_\e f_{\e,*}(e^{-iv'\xi}-e^{-iv\xi})\ud v\ud v_*\ud\theta\\
&=\int_{-\frac{\pi}{2}}^{\frac{\pi}{2}}\beta_\e(\theta) \left( \hat f_\e(\xi\cos\theta)\hat f_\e(\xi\sin\theta) -\hat f_\e(\xi) \hat f_\e(0) \right) \ud\theta.
\end{split}
\eee
The second part can be evaluated using the inverse Fourier transform of the function ${f_\e}$ which is supposed to rely in $L^2$
\bee
\begin{split}
&\delta\int_{-\frac{\pi}{2}}^{\frac{\pi}{2}}\int_{\R^2}\beta_\e(\theta)f_\e f_{\e,*} f_\e(v')(e^{-iv'\xi}-e^{-iv\xi})\ud v\ud v_*\ud\theta\\
&= \frac{\delta}{2\pi}\int_{-\frac{\pi}{2}}^{\frac{\pi}{2}}\beta_\e(\theta)\int_{\R^3}f_\e f_{\e,*} \hat  f_\e(\eta) e^{i\eta(v\cos\theta
-v_*\sin\theta)}\cdot \\ &
\cdot\left( e^{-i\xi(v\cos\theta -v_*\sin\theta)}-e^{-iv\xi} \right)\ud\eta\ud v\ud v_*\ud\theta\\
&=\frac{\delta}{2\pi}\int_{-\frac{\pi}{2}}^{\frac{\pi}{2}}\beta_\e(\theta)\int_\R
\hat f_\e((\xi-\eta)\cos\theta) \hat f_\e((\xi-\eta)\sin\theta)\hat f_\e(\eta) \ud\eta\ud\theta \\
&\qquad\qquad -\frac{\delta}{2\pi}\int_{-\frac{\pi}{2}}^{\frac{\pi}{2}}\beta_\e(\theta)
\int_\R \hat f_\e(\xi-\eta\cos\theta) \hat f_\e(\eta\sin\theta)\hat f_\e(\eta) \ud\eta\ud\theta .
\end{split}
\eee
The third term can be computed in the same way. At the end we get that the Fourier transform of $f_\e$ satisfies

\begin{equation}
\label{Nfourier}
\begin{split}
\frac{\partial \hat f_\e(t,\xi)}{\partial t} &= \int_{-\frac{\pi}{2}}^{\frac{\pi}{2}}
\beta_\e(\theta) \left( \hat f_\e(\xi\cos\theta)\hat f_\e(\xi\sin\theta) - \hat f_\e (\xi) \hat f_\e(0)\right)\ud\theta\\
&+\frac{\delta}{2\pi}\int_{-\frac{\pi}{2}}^{\frac{\pi}{2}}\beta_\e(\theta)
\int_\R \bigg[\hat f_\e((\xi-\eta)\cos\theta) \hat f_\e((\xi-\eta)\sin\theta)\\
&\qquad\qquad\qquad- \hat f_\e(\xi-\eta\cos\theta) \hat f_\e(\eta\sin\theta)\bigg]
\hat f_\e(\eta) \ud\eta\ud\theta\\
&+\frac{\delta}{2\pi}\int_{-\frac{\pi}{2}}^{\frac{\pi}{2}}\beta_\e(\theta)
\int_\R \bigg[\hat f_\e(\xi\cos\theta-\eta\sin\theta) \hat f_\e(-\xi\sin\theta-\eta\cos\theta)\\
&\qquad\qquad\qquad- \hat f_\e(\xi-\eta\sin\theta) \hat f_\e(-\eta\cos\theta)\bigg]\hat f_\e(\eta) \ud\eta\ud\theta.
\end{split}
\end{equation}

To pass to the limit $\e\to 0$, we need some regularity on $\hat f_\e$, or equivalently, $f_\e$ must have bounded moments. The following result will be proved at the end of Section \ref{grazing}.

\begin{thm}
\label{grazingformal} Assume that the cross-section $\beta$ satisfies (H1), (H2), with $1<\nu<\sqrt{5}-1$. Assume that there exists a solution
$f_\e$ to the problem \eqref{BBE2e}. Assume that $f_\e$ is regular, in the sense that $\|f_\e(t)\|_{L^\infty(\R)}$ and
$\|f_\e(t)\|_{H^{\frac{\nu}{2}}}$ are in $L^p(\R_+)$ for some $p$ big enough; then there exists some constants $\lambda, T>0$ independent of $\e$
such that
$$
\sup_{0\le t<T}\|f_\e(t)\|_{L^1_4} \leq \max \left\{ \lambda, \|f_0\|_{L^1_4}\right\}.
$$
\end{thm}

This result formally shows that the Fourier transform of $f_\e$ (in the velocity variable) is four times derivable, with bounded derivatives. It
is then possible to use Taylor expansions in the formulation \eqref{Nfourier} :

 \be \label{Nfourierdvpt}
\begin{split}
&\frac{\partial \hat f_\varepsilon (t,\xi)}{\partial t} \\
&= \frac{1}{\varepsilon^2}\int_{-\frac{\pi}{2}}^{\frac{\pi}{2}}\beta(\theta)
\Bigg[ \varepsilon\theta\xi \Bigg(\hat f_\varepsilon (\xi) \hat f_\varepsilon'(0)\\
&\qquad-\frac{\delta}{2\pi}\int_\R\Big(\hat f_\varepsilon(\xi-\eta)\hat f_\varepsilon(\eta)\hat f_\varepsilon'(0)
 -\hat f_\varepsilon'(-\eta) \hat f_\varepsilon(\eta)\hat f_\varepsilon(\xi) \Big)\ud\eta\Bigg)\\
&\quad+ \varepsilon^2 \frac{\theta^2}{2}\Bigg( \xi^2 \hat f_\varepsilon''(0) \hat f_\varepsilon(\xi)
-\xi \hat f_\varepsilon'(\xi) \hat f_\varepsilon(0) -\frac{\delta}{2\pi}\int_\R
\Big(\hat f_\varepsilon'(\xi-\eta) \hat f_\varepsilon(\eta) \hat f_\varepsilon(0)\\
&\qquad \qquad+ ((\eta-\xi)^2-\eta^2)\hat f_\varepsilon(\xi-\eta)\hat f_\varepsilon(\eta)\hat f_\varepsilon''(0)
 +\xi \hat f_\varepsilon'(\xi) \hat f_\varepsilon(\eta) \hat f_\varepsilon(-\eta)  \\
&\qquad \qquad+ \xi^2 \hat f_\varepsilon''(-\eta)\hat f_\varepsilon(\eta)\hat f_\varepsilon(\xi)
+2\eta \xi \hat f_\varepsilon'(\xi)\hat f_\varepsilon(\eta) \hat f_\varepsilon'(-\eta)\Big)\ud\eta\Bigg) + \theta^2 O(\varepsilon^3)\Bigg]\ud\theta.
\end{split}
\ee
Letting $\e$ go to $0$, we obtain the grazing collision limit, and we recognize in the limit the Fourier form of a quantum Fokker-Planck equation:

\begin{thm}
 Let $(f_\e)$ be a family of solutions to \eqref{BBE2e} where the cross-section  $\beta$ satisfies (H1), (H2), with $1<\nu<\sqrt{5}-1$. Assume that the solutions $f_\e$ are regular enough ($L^\infty((0,T);L^1\cap L^\infty(\R))$ should be enough thanks to Theorem \ref{regularity}). Then there exists a distribution $f$ such that
\[
 \hat f_\e \to \hat f\quad \textrm{when}~~\e\to 0
\]
in $L^\infty((0,T); W^{4,\infty}(\R))$, and $f$ is a solution to
\be
\begin{split}
\label{FKBBE2} \frac{\partial f}{\partial t}=  \left(\int_\R f(v)\ud v\right) \frac{\partial} {\partial v}(vf(1+\delta  f)) +\left( \int_\R
v^2 f(v)(1+\delta f(v))\ud v \right) \frac{\partial^2 f}{\partial v^2}.
\end{split}
\ee
\end{thm}

The study of equation \eqref{FKBBE2} should be close to what is done in \cite{CLR08, CRS06}.

\section{Existence theorems for the mollified model}
\label{existence}

To give rigorous results, we will from now on work on the regularized model
\begin{equation}
  \label{BBEmod2}
\begin{cases}
  & \displaystyle{\frac{\partial f}{\partial t} = \tilde Q_{QBE}(f)}, \qquad t\in \R_+, v\in\R \\
  & f(0,v) = f_0(v)
\end{cases}
\end{equation}
with

\begin{equation}
  \label{Kernelmod2}
\begin{split}
\tilde Q_{QBE}(f)(t,v) = \int_{-\frac{\pi}{2}}^{\frac{\pi}{2}} \beta(\theta)\int_{\R}&\big( f'f_*'(1+\delta \tilde f)(1+\delta \tilde f_*)\\ & -
ff_*(1+\delta  \tilde f')(1+\delta\tilde f_*')\big)\ud v_* \ud\theta.
\end{split}
\end{equation}
The goal of this Section is to prove the existence of a solution to the problem \eqref{BBEmod2}. To start with, we first consider the case of a
cross-section with cutoff.

%We have

\begin{thm}
\label{existencecutoff}
 Let $f_0 \in L^1_2(\R)$ be a nonnegative function. Assume that the cross-section satisfies $\beta \in L^1(-\frac{\pi}{2};
\frac{\pi}{2})$ and (H1). Then there exists a unique solution $f \in L^\infty(\R_+;L^1_2(\R))$ to the problem \eqref{BBEmod2}, which is
nonnegative, and preserves mass and energy.
\end{thm}

The proof is a consequence of the following theorem from Wild \cite{Wild}:

\begin{thm}
\label{ThmWild}
 Let $E$ be a given Banach space and $P:E^N\to E$ (for $N\ge 2$) be a N-linear operator satisfying the inequality
\[
 \|P(u_1,\cdots,u_N)\|\le C_p \|u_1\|\cdots\|u_N\|\quad\textrm{for all}~u_i \in E.
\]
Let
\[
 u(t)=\sum_{k=0}^{\infty}b_k e^{-t}(1-e^{(1-N)t})^k u_k,
\]
where
\[
 u_k= \sum_{i_1+\cdots +i_N=k=1}\frac{b_{i_1}\cdots b_{i_N}}{k(N-1)b_k}P(u_{i_1},\cdots,u_{i_N})\quad \textrm{for}~k\ge 1
\]
and the $b_k$ are the coefficients of the Taylor expansion of the function
\[
 v(x)-(1-x)^{\frac{1}{1-N}}= \sum_{k=0}^\infty b_k x^k.
\]
Let
\[
 t_0= \frac{1}{1-N}\log (1+C_p^{-1}\|u_0\|^{1-N})
\]
and
\bee
t_1=
\begin{cases}
 &\frac{1}{1-N}\log (1-C_p^{-1}\|u_0\|^{1-N})\qquad \textrm{if}~1>C_p^{-1}\|u_0\|^{1-N} \\
&\infty  \qquad\qquad\qquad \textrm{if}~1\le C_p^{-1}\|u_0\|^{1-N}.
\end{cases}
\eee
Then $u(t)$ is uniformly convergent on compact subsets of $(t_0,t_1)$ and is the solution of the equation
\[
 \frac{\ud u}{\ud t}= -u + P(u,\cdots,u),
\]
with the initial condition
\[
 u(0)=u_0.
\]

\end{thm}

\begin{proof}[Proof of Theorem \ref{existencecutoff}]
For $f,g,h \in L^1_2(\R)$, let
\bee
\begin{split}
P(f,&g,h) = \int_{-\frac{\pi}{2}}^{\frac{\pi}{2}}\beta(\theta) \int_\R \Bigg[ f'g_*'
\left( \frac{\int_\R h}{\|f_0\|_{L^1}} +\delta \tilde h +\delta \tilde h_* \right) -fg_* \bigg( \frac{\int_\R h}{\|f_0\|_{L^1}} \\
&+\delta \tilde h' +\delta \tilde h_*' \bigg)\Bigg]\ud v_* \ud\theta + f \left(\int_\R g\right)
 \left(  \frac{\int_\R h}{\|f_0\|_{L^1}} + 2\delta \|\psi\|_{L^\infty} \int_\R h \right)\int_{-\frac{\pi}{2}}^{\frac{\pi}{2}} \beta(\theta)\ud\theta.
\end{split}
\eee
Let
$$
K = \|f_0\|_{L^1} \left( 1+2\delta \|\psi\|_{L^\infty} \|f_0\|_{L^1} \right) \int_{-\frac{\pi}{2}}^{\frac{\pi}{2}} \beta(\theta)\ud\theta.
$$
Let us consider the following problem
\begin{equation}
\label{PbWild}
  \begin{cases}
    & \displaystyle{\frac{\partial f}{\partial t}+ Kf = P(f,f,f)}, \qquad t\in\R_+, v\in\R\\
    & f(0,v) = f_0.
  \end{cases}
\end{equation}
The operator $P: (L^1_2)^3 \to L^1_2$ is trilinear, and satisfies the inequality
$$
\|P(f,g,h)\|_{L^1_2} \leq C_P \|f\|_{L^1_2}\|g\|_{L^1_2}\|h\|_{L^1_2} \qquad \forall f,g,h \in L^1_2 (\R),
$$
with
$$
C_P = \int \beta(\theta)\ud\theta \left(\frac{2}{\|f_0\|_{L^1}}+4\delta \|\psi\|_{L^\infty}\right).
$$
Assume for the moment that $K =1$.  Thanks to Theorem \ref{ThmWild}, there exists some $T>0$ such that there exists a solution $f\in L^\infty(
0,T ; L^1_2(\R))$ to the problem \eqref{PbWild}. This solution can be written as a Wild sum, which reads
$$
f(t) = \sum_{k=0}^{+\infty} b_k e^{-t} (1-e^{-2t})^kf_k,
$$
where

$$
f_k = \sum_{i_1+i_2+i_3=k-1} \frac{b_{i_1}b_{i_2}b_{i_3}}{2kb_k}P(f_{i_1},f_{i_2},f_{i_3})\qquad \textrm{for}~ k\geq 1
$$
and
$$
f_{k=0} = f_0.
$$
The numbers $b_k$ are the coefficients of the Taylor expansion of
$$
\frac{1}{\sqrt{1-x}} = \sum_{k=0}^{+\infty} b_k x^k.
$$
One can easily see that all the $b_k$ are positive. Moreover
$$
0\leq f,g,h \in L^1_2 \implies P(f,g,h)\geq 0.
$$
Thus, the solution $f$ of \eqref{PbWild} is nonnegative. Moreover, owing to the definition of $P(\cdot,\cdot,\cdot)$,  one can verify that this
solution preserves the mass. From that we deduce that $f$ is solution of \eqref{BBEmod2} on $(0,T)$. But it preserves the mass, and it relies in
$L^1_2$, so that it also preserves the energy. Since the time $T$ depends only on $\|f_0\|_{L^1}, \|f_0\|_{L^1_2}, \|\psi\|_{L^\infty},
\|\beta\|_{L^1(-\frac{\pi}{2},\frac{\pi}{2})}$, we can use the same arguments on the time intervals $(T,2T)$, $(2T,3T)$, etc, to get a solution on
$(0,+\infty)$. Finally, in the case $K\neq1$, it is enough to rescale the time to get the right formulae, and to obtain the same conclusions.
\end{proof}

The following theorem claims the existence of a solution of \eqref{BBEmod2} in the non-cutoff case in some weak sense.

\begin{thm}
\label{thmexistence} Let $f_0 \in L^1_2(\R)$ be a nonnegative function. Let $\beta$ satisfy the assumptions (H1) and (H2). Then, there exists a
distribution $g$ such that its Fourier transform $\hat g\in L^\infty(\R_+; C^0\cap L^\infty( \R))$  is a solution of \eqref{fourier}, the Fourier form of equation \eqref{BBEmod2}. Moreover, $\hat g$ preserves the mass.
\end{thm}

\begin{proof}
We introduce, for $n\in\N_*$, the cross-section
$$
\beta_n(\theta) = \beta(\theta)\wedge n = \min (\beta(\theta), n),
$$
and we denote by $f^n$ the solution of the problem \eqref{BBEmod2} corresponding to the cross-section $\beta_n$. This solution exists thanks to
Theorem \ref{existencecutoff}. For all $n$ and all $t>0$, $f_n(t)$ relies in $L^1$, so that we can define its Fourier transform. Moreover, $\tilde
Q_{QBE}(f^n)$ also relies in $L^1$. Hence, we can write the following equation: \bee
\begin{split}
\frac{\partial \hat f^n(t,\xi)}{\partial t} = \int_{-\frac{\pi}{2}}^{\frac{\pi}{2}}\int_{\R^2}\beta_n(\theta)f^n f^n_*
(1+\delta \tilde {f^n}(v')+\delta \tilde {f^n}(v_*'))(e^{-iv'\xi}-e^{-iv\xi})\ud v\ud v_*\ud\theta.
\end{split}
\eee
As in Section \ref{formal} and since $\mathcal{F}(\tilde f^n)(\eta) = \hat f^n(\eta) \hat\psi(\eta)$, we get

\begin{equation}
\label{fourier}
\begin{split}
\frac{\partial \hat f^n(t,\xi)}{\partial t} &= \int_{-\frac{\pi}{2}}^{\frac{\pi}{2}}
\beta_n(\theta) \left( \hat f^n(\xi\cos\theta)\hat f^n(\xi\sin\theta) - \hat f^n (\xi) \hat f^n(0)\right)\ud\theta\\
&+\frac{\delta}{2\pi}\int_{-\frac{\pi}{2}}^{\frac{\pi}{2}}\beta_n(\theta)
\int_\R \bigg[\hat f^n((\xi-\eta)\cos\theta) \hat f^n((\xi-\eta)\sin\theta)\\
&\qquad\qquad\qquad- \hat f^n(\xi-\eta\cos\theta) \hat f^n(\eta\sin\theta)\bigg]
\hat f^n(\eta) \hat\psi(\eta) \ud\eta\ud\theta\\
&+\frac{\delta}{2\pi}\int_{-\frac{\pi}{2}}^{\frac{\pi}{2}}\beta_n(\theta)
\int_\R \bigg[\hat f^n(\xi\cos\theta-\eta\sin\theta) \hat f^n(-\xi\sin\theta-\eta\cos\theta)\\
&\qquad\qquad\qquad- \hat f^n(\xi-\eta\sin\theta) \hat f^n(-\eta\cos\theta)\bigg]\hat f^n(\eta) \hat\psi(\eta) \ud\eta\ud\theta.
\end{split}
\end{equation}
Note that here the mass is the quantity $\hat f(0)$. Since the second moment of $f^n(v,t)$ is finite and conserved in time, its Fourier transform is two times differentiable, and satisfies
$$
\sup_{t>0}\|\partial^2_{\xi\xi}\hat f^n (t,\xi)\|_{L^\infty}\leq \int_\R v^2f_0(v)\ud v.
$$
Hence, we can use the Taylor formula at the order $2$
$$
h(\theta)=h(0)+\theta h'(0) +\theta^2 \int_0^1 (1-s) h''(s\theta)\ud s
$$
on the functions $\theta \mapsto \hat f^n(\xi\cos\theta)\hat f^n(\xi\sin\theta)$, $\theta \mapsto \hat f^n((\xi-\eta)\cos\theta) \hat
f^n((\xi-\eta)\sin\theta)$, $\theta \mapsto \hat f^n(\xi-\eta\cos\theta) \hat f^n(\eta\sin\theta)$, $\theta \mapsto \hat
f^n(\xi\cos\theta-\eta\sin\theta) \hat f^n(-\xi\sin\theta-\eta\cos\theta)$ and $\theta\mapsto \hat f^n(\xi-\eta\sin\theta) \hat
f^n(-\eta\cos\theta)$.

Using the notations $m= \int_\R f_0(v)\ud v$ and $e=\int_\R v^2 f_0(v)\ud v$, we get the following estimates
 \be \label{estim1}
\begin{split}
&\left|\int_{-\frac{\pi}{2}}^{\frac{\pi}{2}}\beta_n(\theta) \left( \hat f^n(\xi\cos\theta)
\hat f^n(\xi\sin\theta) - \hat f^n (\xi) \hat f^n(0)\right)\ud\theta\right|\\
&\qquad\qquad\leq \int_{-\frac{\pi}{2}}^{\frac{\pi}{2}}\theta^2\beta(\theta)\ud\theta \left( 4|\xi|^2 me+2|\xi|m^{3/2}e^{1/2}\right),
\end{split}
\ee

\begin{equation}
  \label{estim2}
\begin{split}
&\Bigg|\int_{-\frac{\pi}{2}}^{\frac{\pi}{2}}\beta_n(\theta)\bigg( \hat f^n((\xi-\eta)\cos\theta) \hat f^n((\xi-\eta)\sin\theta)\\
&\qquad\qquad\qquad\qquad\qquad\qquad- \hat f^n(\xi-\eta\cos\theta) \hat f^n(\eta\sin\theta) \bigg)\mathcal{F}(\tilde f^n)(\eta)\ud\theta\Bigg|\\
&\qquad\leq |\hat\psi(\eta)|m\int_{-\frac{\pi}{2}}^{\frac{\pi}{2}}
\theta^2\beta(\theta)\ud\theta\left( 4(|\xi-\eta|^2+|\eta|^2)me + 2 (|\xi-\eta|+|\eta|)m^{3/2}e^{1/2} \right)
\end{split}
\end{equation}
and

\begin{equation}
  \label{estim3}
\begin{split}
&\Bigg|\int_{-\frac{\pi}{2}}^{\frac{\pi}{2}}\beta_n(\theta)\bigg( \hat f^n(\xi\cos\theta-\eta\sin\theta) \hat f^n(-\xi\sin\theta -\eta\cos\theta)\\
&\qquad\qquad\qquad\qquad\qquad- \hat f^n(\xi-\eta\sin\theta) \hat f^n(-\eta\cos\theta) \bigg)\mathcal{F}(\tilde f^n)(\eta)\ud\theta\Bigg|\\
&\qquad\leq |\hat\psi(\eta)|m\int_{-\frac{\pi}{2}}^{\frac{\pi}{2}}\theta^2\beta(\theta)\ud\theta
\left( 4((|\xi|+|\eta|)^2+|\eta|^2)me + 2 (|\xi|+2|\eta|)m^{3/2}e^{1/2} \right)
\end{split}
\end{equation}
The right member of \eqref{estim1} is integrable in time on any interval $[t_1,t_2]\subset \R_+$, and the right-members of \eqref{estim2} and
\eqref{estim3} are integrable in $(t,\eta)$ on any $[t_1,t_2]\times \R$ with $0<t_1<t_2$, since $\psi \in C_c^\infty(\R)$. Therefore, to pass to
the limit in \eqref{fourier}, it is enough for $\hat f^n$ to converge pointwise on $\R_+\times \R$. But inequalities \eqref{estim1},
\eqref{estim2} and \eqref{estim3} ensure that for all compact set $\mathcal K \subset \R$, there exists a constant C depending only on $\mathcal
K$, $m$, $e$, $\psi$, $\beta$, such that
$$
|\hat f^n(t_1,\xi)-\hat f^n(t_2,\xi)|\leq C |t_1-t_2|\qquad\forall \, 0<t_1<t_2,\quad\forall \, \xi\in\mathcal K.
$$
Then, thanks to Ascoli's theorem, there exists a function $\hat g \in L^\infty([t_1,t_2]\times \mathcal K)$ such that, up to the extraction of a
subsequence,
$$
\|\hat f^n -\hat g\|_{L^\infty([t_1,t_2]\times \mathcal K)} \underset{n\to+\infty}{\longrightarrow}0.
$$
All this being true for every $t_1$, $t_2$, $\mathcal K$, we deduce that $\hat g$ is well defined on $\R_+\times\R$ and that $\hat g\in L^\infty(\R_+;
L^\infty\cap C^0(\R))$. We can therefore pass to the limit in \eqref{fourier}. Finally, we obtained the existence of a function $\hat g(\xi,t)$ which
satisfies \fer{fourier} with the original cross-section $\beta$, and such that
$$
\hat g(t,0)=m \qquad\forall t>0.
$$

%Now, recording than $f^n \in L^\infty(\R_+; L^1_2(\R))$, there exists a function $f\in L^\infty(\R_+;\mathcal
%M_2)$ such that, up to the extraction of a subsequence, we have, for all $\phi$ verifying $\frac{\phi(v)}{1+v^2}\in L^\infty$,
%$$
%\int_\R f^n(t,v)\phi(v)\ud v \underset{n \to +\infty}{\longrightarrow}\int_\R \phi(v)f(t,v)\ud v.
%$$
%Therefore, we get that $g= \hat f$ and that $f$ preserves mass and energy.

\end{proof}

\section{Moments of the cutoff solutions, regularity of the non cutoff solution}
\label{sectionmoments}

The second step of our analysis is the study of the regularity of the Fourier transform of the solution obtained in the previous Section. This can be done by investigating the moments of the solution to the cut-off equation. Since it is enough for our needs, we will limit ourselves to the fourth moment. Let
\[
 A=\int_{-\pi/2}^{\pi/2}\b (\t) \cos\t\sin^2\t\ud\t
\]
\[
 A_\eta=\int_{-\pi/2}^{\pi/2}\b (\t) \cos\t|\sin\t|^{1+\eta}\ud\t,\qquad \nu-1<\eta<1,
\]
\[
 A_*=\int_{-\pi/2}^{\pi/2}\b (\t) \cos^2\t\sin^2\t\ud\t
\]
and chose $n_\b\in\N$ large enough such that, for all $n\ge n_\b$,
\[
 \int_{-\pi/2}^{\pi/2}\b_n (\t) \cos^2\t\sin^2\t\ud\t \ge \frac{A_*}{2}
\]
with $\b_n(\t)=\min (\b(\t),n)$. Our result is the following

\begin{thm}
\label{moments}
 Let $\beta(\theta)$ satisfy the assumptions (H1) and (H2), and let $f^n(t,v)$ be the solution of the problem \eqref{BBEmod2} with
cross-section $\beta_n$ and with nonnegative initial datum $f_0 \in L^1_4 (\R)$. Then
there are explicit constants $a>0$, $c>0$, $C>0$ such that we have
\be
\label{inmom2}
\sup_{n\ge n_\b} \|f^n(t)\|_{L^1_4} \leq (\|f_0\|_{L^1_4}+ct)e^{at}, \quad t\ge0,
\ee

\be
\label{inmom}
\sup_{n\ge n_\b}\sup_{t\in \R_+} \|f^n(t)\|_{L^1_4} \leq \max \{C,\|f_0\|_{L^1_4} \}.
\ee
Moreover, $a,c$ in \eqref{inmom2} depend only on $\|f_0\|_{L^1}$, $\|f_0\|_{L^1_2}$, $\delta$, $\psi$ and A, whereas $C$ in \eqref{inmom} depends only on $\|f_0\|_{L^1}$, $\|f_0\|_{L^1_2}$, $\delta$, $\psi$ and A, $A_*$, $A_\eta$ and $\eta$.
\end{thm}

\begin{proof}
Let us take $\phi(v)=v^4$ as test function. Thanks to the symmetries of the kernel we obtain
 \bee
\begin{split}
&\frac{\ud}{\ud t}\int_\R v^4f^n(v,t)\ud v \\= &\frac{1}{2}\int_{-\frac{\pi}{2}}^{\frac{\pi}{2}}\beta_n(\theta)\int_{\R^2}(v'^4-v^4+v_*'^4-v_*^4)
f^nf^n_*(1+\delta \tilde f^n(v') +\delta \tilde f^n(v_*'))\ud v \ud v_* \ud\theta.
\end{split}
\eee

From the collision rule \fer{collision} it follows that
 \bee
\begin{split}
&(v')^4+(v_*')^4-v^4-v_*^4 \\
&= (v^4+v_*^4)(\cos^4\theta+\sin^4\t -1)+12v^2v_*^2\sin^2\t\cos^2\t \\
&\quad+4vv_*\cos\t\sin\t(v^2(\sin^2\t-\cos^2\t)+v_*^2(\cos^2\t-\sin^2\t))\\
&=-2\cos^2\t\sin^2\t(v^4+v_*^4)+12v^2v_*^2\sin^2\t\cos^2\t\\
&\quad +4 \cos\t\sin\t\cos(2\t)vv_*(v_*^2-v^2).
\end{split}
 \eee
Consequently
\bee
\begin{split}
 &\frac{\ud}{\ud t}\|f^n(t)\|_{L^1_4}\\
&=-\int_{-\pi/2}^{\pi/2}\b_n(\t)\cos^2\t\sin^2\t \iint_{\R^2}\left(v^4+v_*^4\right)f^nf_*^n\left(1+\delta \tilde f^n(v')+\delta \tilde f^n(v_*')\right)\ud v\ud v_*\ud\t\\
&+6\int_{-\pi/2}^{\pi/2}\b_n(\t)\cos^2\t\sin^2\t \iint_{\R^2}v^2v_*^2f^nf_*^n\left(1+\delta \tilde f^n(v')+\delta \tilde f^n(v_*')\right)\ud v\ud v_*\ud\t\\
&+2\int_{-\pi/2}^{\pi/2}\b_n(\t)\cos\t\sin\t\cos(2\t)\iint_{\R^2}vv_*f^nf_*^n\left(1+\delta \tilde f^n(v')+\delta \tilde f^n(v_*')\right)\ud v\ud v_*\ud\t\\
&\coloneq I_1+I_2+I_3.
\end{split}
\eee

In what follows we always assume $n\ge n_\b$. For the first and second terms $I_1$, $I_2$ we use
\[
 \frac{1}{2}A_*\le \int_{-\pi/2}^{\pi/2}\b_n(\t)\cos^2\t\sin^2\t\ud\t \le A_*
\]
and
\[
 \|\tilde f^n(t)\|_{L^\infty}\le m\|\psi\|_{L^\infty}
\]
to get
\be
\label{*3}
 \begin{split}
  I_1&\le -\left(\frac{1}{2}\int_{-\pi/2}^{\pi/2}\b(\t)\cos^2\t\sin^2\t\ud\t \right)2m\int_{\R}|v|^4f^n(t,v)\ud v\\
 &=m^2A_*-mA_*\|f^n(t)\|_{L^1_4},
 \end{split}
\ee
and
\be
\label{*4}
 I_2\le 6(1+2\delta m\|\psi\|_{L^\infty})A_*e^2.
\ee
For the third term $I_3$ we compute by change of variable $(v,v_*,\t)\mapsto (v,v_*,-\t)$
\[
 \begin{split}
  &I_3=\delta\int_{-\pi/2}^{\pi/2}\b_n(\t)\cos\t\sin\t\cos(2\t)\\
&\times \iint_{\R^2}vv_*(v_*^2-v^2)f^nf_*^n\left(\tilde f^n(v') +\tilde f^n(v_*') -\tilde f^n(\bar v')-\tilde f^n(\bar v_*')\right) \ud v\ud v_*\ud\t
 \end{split}
\]
where
\[
 \bar v'=v\cos\t +v_*\sin\t,\qquad \bar v_*'=-v\sin\t+v_*\cos\t.
\]
Recall
\[
 v'=v\cos\t-v_*\sin\t,\qquad v_*'=v\sin\t+v_*\cos\t
\]
and
\[
 \|\frac{\partial}{\partial v}\tilde f^n(.,t)\|_{L^\infty}\le m\|\psi'\|_{L^\infty}.
\]
It follows that
\be
\label{*5}
 \begin{split}
  |\tilde f^n(v')+\tilde f^n(v_*') -\tilde f^n(\bar v')-\tilde f^n(\bar v_*')|&\le |\tilde f^n(v')-\tilde f^n(\bar v')|+ |\tilde f^n(v_*')-\tilde f^n(\bar v_*')|\\
&\le 2m\|\psi'\|_{L^\infty}(|v|+|v_*|)|\sin\t|.
 \end{split}
\ee
The left hand side is also less than $2m\|\psi\|_{L^\infty}$ (since $f\ge 0$). Thus
\be
\label{*6}
  |\tilde f^n(v')+\tilde f^n(v_*') -\tilde f^n(\bar v')-\tilde f^n(\bar v_*')|\le 2m(\|\psi\|_{L^\infty})^{1-\eta}(\|\psi'\|_{L^\infty})^{1-\eta} (|v|^\eta+|v_*|^\eta)|\sin\t|^\eta.
\ee
By \eqref{*5} and \eqref{*6} we obtain two estimates for $I_3$:
\be
\label{*7}
\begin{split}
 I_3\le 2m\delta\|\psi'\|_{L^\infty}A&\iint_{\R^2}|v||v_*|(|v|+|v_*|)(v_*^2+v^2)f^nf^n_*\ud v\ud v_*\\
&=K_{m,e}\delta\|\psi'\|_{L^\infty}A\|f^n(t)\|_{L^1_4},
\end{split}
\ee
and
\be
\label{*8}
 \begin{split}
  I_3&\le 2m(\|\psi\|_{L^\infty})^{1-\eta}(\|\psi'\|_{L^\infty})^{1-\eta}\delta A_\eta \iint_{\R^2}|v||v_*|(|v|^\eta +|v_*|^\eta)(v_*^2+v^2)f^nf^n_* \ud v\ud v_*\\
&\le K_{m,e}(\|\psi\|_{L^\infty})^{1-\eta}(\|\psi'\|_{L^\infty})^{1-\eta}\delta A_\eta(\|f^n(t)\|_{L^1_4})^{\frac{3+\eta}{4}} \coloneq C_1(\|f^n(t)\|_{L^1_4})^{\frac{3+\eta}{4}}.
 \end{split}
\ee
Here $K_{m,e}>0$ depends only on $m$ and $e$.

To prove the fist estimate \eqref{inmom2} we omit the negative term $I_1$ and use \eqref{*4}, \eqref{*7} and notice that $A_*\le A$ to get
\[
 \frac{\ud}{\ud t}\|f^n(t)\|_{L^1_4}\le 6e^2(1+2\delta m\|\psi\|_{L^\infty})A +K_{m,e}\delta\|\psi'\|_{L^\infty}A\|f^n(t)\|_{L^1_4},\quad t>0.
\]
This gives by Gronwall lemma
\[
 \sup_{n\ge n_\b} \|f^n(t)\|_{L^1_4}\le (\|f_0\|_{L^1_4}+ct)e^{at},\quad t\ge 0
\]
where
\[
 c=6e^2(1+2\delta m\|\psi\|_{L^\infty})A,\qquad a=K_{m,e}\delta\|\psi'\|_{L^\infty}A.
\]

To prove the second estimate \eqref{inmom} we use \eqref{*3}, \eqref{*4} and \eqref{*8} to see that
\[
\frac{\ud}{\ud t}\|f^n(t)\|_{L^1_4} \le C_2 + C_1(\|f^n(t)\|_{L^1_4})^{\frac{3+\eta}{4}} -mA_*\|f^n(t)\|_{L^1_4}, \quad t>0.
\]
Here and below $C_1$, $C_2$ and $C_3$ depend only on $m,e,\psi,\delta,A,A_*,A_\eta$ and $\eta$. Applying the following inequality
\[
 Y^\alpha \le (1/\e)^{\frac{\alpha}{1-\alpha}}+\e Y,\qquad \alpha \in (0,1), Y\ge 0,\e >0
\]
to
\[
 Y=\|f^n(t)\|_{L^1_4},\quad \alpha=\frac{3+\eta}{4},\quad \e=\frac{mA_*}{2C_1}
\]
gives
\[
 C_1(\|f^n(t)\|_{L^1_4})^\frac{3+\eta}{4}\leq C_1 \left(\frac{2C_1}{mA_*}\right)^\frac{3+\eta}{1-\eta}+ \frac{mA_*}{2}\|f^n(t)\|_{L^1_4}
\]
and so
\[
 \frac{\ud}{\ud t}\|f^n(t)\|_{L^1_4}\le C_3-\frac{mA_*}{2}\|f^n(t)\|_{L^1_4},\qquad t>0.
\]
Therefore,
\[
 \sup_{n\ge n_\b} \sup_{t\ge 0} \|f^n(t)\|_{L^1_4}\le \max \left \{ \frac{2C_3}{mA_*},\|f_0(t)\|_{L^1_4}\right\}.
\]

\end{proof}

Passing to the limit $n \to +\infty$ in inequality \eqref{inmom} we obtain

\begin{thm}
\label{regularite} Let $\beta$ satisfy (H1) and (H2), and let $g$ be the weak solution of \eqref{BBEmod2} defined in Theorem \ref{thmexistence},
with nonnegative initial data $f_0 \in L^1_4$. Then, $\hat g(t)$ is $C^3$ for all $t$ and
$$
\sup_{t>0}\{\|\hat g(t)\|_{L^\infty}+\|\partial_\xi \hat g(t)\|_{L^\infty}+\|\partial^2_{\xi^2}\hat g(t)\|_{L^\infty} +
 \|\partial^3_{\xi^3}\hat g(t)\|_{L^\infty} + \|\partial^4_{\xi^4}\hat g(t)\|_{L^\infty}\}<+\infty.
$$
Moreover, $\hat g$ conserves the energy, in the sense that $\partial^2_{\xi\xi} \hat g(t,0) = -e$ for all $t>0$.
\end{thm}

\begin{proof}
The conservation of the mass and inequality \eqref{inmom} imply that there exists a constant $C>0$ which do not depend on $n$ such that
\bee
\begin{cases}
  & \|\partial_\xi \hat f^n\|_{L^\infty(\R_+\times\R)} \leq C \\
  & \|\partial^2_{\xi^2} \hat f^n\|_{L^\infty(\R_+\times\R)} \leq C \\
  & \|\partial^3_{\xi^3} \hat f^n\|_{L^\infty(\R_+\times\R)} \leq C \\
  & \|\partial^4_{\xi^4} \hat f^n\|_{L^\infty(\R_+\times\R)} \leq C .
\end{cases}
\eee
Since $L^\infty(\R_+\times\R)$ is the dual space of the Banach space $L^1(\R_+\times\R)$, the four sequences converge (up to the extraction
of a subsequence) in $L^\infty(\R_+\times\R)$ weak-*; the limits can only be respectively $\partial_\xi \hat g$, $\partial^2_{\xi^2} \hat g$,
$\partial^3_{\xi^3} \hat g$ and $\partial^4_{\xi^4} \hat g$ (since the convergence in $L^\infty(\R_+\times\R)$ weak-* implies the convergence in the
distributional sense). Moreover, we have the inequalities, for $1\leq i\leq 4$,
$$
\|\partial^i_{\xi^i}\hat g\|_{L^\infty(\R_+\times\R)}\leq \liminf_{n\to +\infty} \|\partial^i_{\xi^i} \hat f^n\|_{L^\infty(\R_+\times\R)}.
$$
Finally, we have the embedding
$$
W^{4,\infty}(\R) \hookrightarrow C^3(\R),
$$
so that $\hat g\in L^\infty(\R_+;C^3(\R))$.

It remains to prove that the energy is conserved. Let us fix some time $t_0 >0$. It is clear that for all integer $n$, we have
$$
\|\partial^2_{\xi^2}\hat f^n(t_0,.)\|_{L^\infty(\R)} \leq e.
$$
Therefore, up to the extraction of a subsequence, there exists a function $h \in L^\infty(\R)$ such that
$$
\partial^2_{\xi^2}\hat f^n(t_0,.) \rightharpoonup h \qquad \textrm{weak}-*~~L^\infty(\R).
$$
But it is clear that $h= \partial^2_{\xi^2}g(t_0,.)$, since for all function $\phi \in C^\infty(\R)$ with compact support,
$$
\int_\R \hat f^n(t_0,\xi)\phi''(\xi)\ud\xi \to \int_\R \hat g(t_0,\xi)\phi''(\xi)\ud\xi,
$$
or, after two integrations per part on both sides,
$$
\int_\R \partial^2_{\xi^2}\hat f^n(t_0,\xi)\phi(\xi)\ud\xi \to \int_\R \partial^2_{\xi^2}\hat g(t_0,\xi)\phi(\xi)\ud\xi.
$$
Let us define an approximation of the Dirac measure
 \bee \Phi_p(\xi)=
\begin{cases}
  &p \qquad\textrm{if}~~-\frac{1}{2p}<\xi<\frac{1}{2p}\\
  &0 \qquad\textrm{otherwise}
\end{cases}
\eee
We have

\bee
\begin{split}
\left|\partial^2_{\xi^2}\hat f^n(t_0,0)-\partial^2_{\xi^2}\hat g(t_0,0)\right|
&\leq \left|\partial^2_{\xi^2}\hat f^n(t_0,0)-\int_\R \Phi_p(\xi)\partial^2_{\xi^2}\hat f^n(t_0,\xi)\ud\xi\right|\\
&\qquad\qquad+ \left|\int_\R \Phi_p(\xi) \left(\partial^2_{\xi^2}\hat f^n(t_0,\xi)- \partial^2_{\xi^2} \hat g(t_0,\xi) \right) \ud \xi\right|\\
&\qquad\qquad+ \left|\int_\R \Phi_p(\xi)\partial^2_{\xi^2} \hat g(t_0,\xi)\ud \xi - \partial^2_{\xi^2}\hat g(t_0,0)\right|.
\end{split}
\eee
The first and the third terms converge toward 0 when $p$ converges to infinity, independently of $n$ since
$$
\|\partial^3_{\xi^3}\hat f^n(t_0,.)\|_{L^\infty} \leq C
$$
with $C>0$ independent of $n$. As for the second term, once $p$ has been fixed, it converges to 0 since $\Phi_p \in L^1(\R)$, and the result follows.
\end{proof}

\section{The grazing collision limit}
\label{grazing}

We are now in a position  to perform the grazing collision limit in equation \eqref{BBEmod}. The mechanism is the same as in Section \ref{formal}: we introduce a
family of kernels $\{ \b_\e(\vt) \}_{\e > 0}$ satisfying hypotheses (H1) and (H2), with
\be
\label{betaepsilon2}
\forall \theta_0>0~~~\sup_{\theta>\theta_0}\b_\e(\vt) \underset{\e\to 0}{\longrightarrow}0
\ee
and
 \be\label{betaepsilon}
\lim_{\e \to 0^+} \int_0^\pi \b_\e(\vt)\t^2 \, d\t = 1
 \ee
Let $g_\varepsilon$ be the weak solution of the problem \eqref{BBEmod2} in the sense that it satisfies equation \eqref{fourier},
where $\beta(\theta)$ has been replaced by $\beta_\varepsilon(\theta)$.

\begin{thm}
Let $\beta(\theta)$ satisfy assumptions (H1), (H2), and let $\beta_\varepsilon(\theta)$ satisfy \eqref{betaepsilon2} and \eqref{betaepsilon}. Let $g_\varepsilon$ be the
weak solution of the problem \eqref{BBEmod2} where $\beta(\theta)$ has been replaced by $\beta_\varepsilon(\theta)$, with the nonnegative initial
data $f_0$  satisfying $f_0 \in L^1_4(\R)$.

Then, for all $T>0$, there exists a distribution $g$ whose Fourier transform satisfy $\hat g \in L^\infty(0,T;W^{4,\infty}(\R))$, and such that, up to the extraction of a subsequence,
$$
\|\hat g_\varepsilon-\hat g\|_{L^\infty([t_1,t_2]\times \mathcal K)} \underset{\varepsilon \to 0} {\longrightarrow }0 \qquad \forall \,
0<t_1<t_2<T~~\textrm{and}~~\mathcal K \subset \R ~~\textrm{compact},
$$
$$
\partial^i_{\xi^i}\hat g_\varepsilon \underset{\varepsilon \to 0}{\rightharpoonup }
\partial^i_{\xi^i}\hat g\quad \textrm{in}~~L^\infty((0,T)\times\R)~~\textrm{weak-*},~~~1\leq i\leq 4,
$$
and such that $\hat g$ is a solution of the Fourier form of the equation
\be
\begin{split}
\label{FKBBE} \frac{\partial h}{\partial t}=  \left(\int_\R h(v)\ud v\right) \frac{\partial} {\partial v}(vh(1+\delta \tilde h)) +\left( \int_\R
v^2 h(v)(1+\delta \tilde h(v))\ud v \right) \frac{\partial^2 h}{\partial v^2},
\end{split}
\ee
\end{thm}

\begin{rem}
If $\delta =0$, equation \eqref{FKBBE} reduces to the classical linear Fokker-Planck equation.
\end{rem}

\begin{proof}
To pass to the limit, we need some regularity on the solution $g_\varepsilon$. However, Theorem \ref{regularite} is not sufficient, since the
bound on the fourth derivative of $\hat g_\varepsilon$ depends on $A_\eta$ which becomes an unbounded quantity when replacing $\beta$ by
$\beta_\varepsilon$ and as $\e\to 0$. Therefore, we use inequality \eqref{inmom2} to obtain a regularity result on $\hat g_\varepsilon$ which is
the same than Theorem \ref{regularite}, except that it works only on any time interval $(0,T)$. The gain is that the constants used remain bounded
as $\e\to 0$ when replacing $\beta$ by $\beta_\varepsilon$.

Let us fix some $T>0$; $\hat g_\varepsilon$ is four times differentiable for almost any time in $(0,T)$, and these derivatives are bounded uniformly in time and independently of $\varepsilon$, provided the initial datum $f_0$ relies in $L^1_4(\R)$. We now act as if $\beta(\theta)$ was integrable, and then we will obtain the result by an approximation argument.  A Taylor expansion in $\varepsilon$ under the integral sign gives
 \be \label{fourierdvpt}
\begin{split}
&\frac{\partial \hat g_\varepsilon (t,\xi)}{\partial t} \\
&= \frac{1}{\varepsilon^2}\int_{-\frac{\pi}{2}}^{\frac{\pi}{2}}\beta(\theta)
\Bigg[ \varepsilon\theta\xi \Bigg(\hat g_\varepsilon (\xi) \hat g_\varepsilon'(0)\\
&\qquad-\frac{\delta}{2\pi}\int_\R\Big(\hat g_\varepsilon(\xi-\eta)\hat g_\varepsilon(\eta)\hat g_\varepsilon'(0)
\hat\psi(\eta) -\hat g_\varepsilon'(-\eta) \hat g_\varepsilon(\eta)\hat g_\varepsilon(\xi)\hat\psi(\eta) \Big)\ud\eta\Bigg)\\
&\quad+ \varepsilon^2 \frac{\theta^2}{2}\Bigg( \xi^2 \hat g_\varepsilon''(0) \hat g_\varepsilon(\xi)
-\xi \hat g_\varepsilon'(\xi) \hat g_\varepsilon(0) -\frac{\delta}{2\pi}\int_\R \hat\psi(\eta)
\Big(\hat g_\varepsilon'(\xi-\eta) \hat g_\varepsilon(\eta) \hat g_\varepsilon(0)\\
&\qquad \qquad+ ((\eta-\xi)^2-\eta^2)\hat g_\varepsilon(\xi-\eta)\hat g_\varepsilon(\eta)\hat g_\varepsilon''(0)
 +\xi \hat g_\varepsilon'(\xi) \hat g_\varepsilon(\eta) \hat g_\varepsilon(-\eta)  \\
&\qquad \qquad+ \xi^2 \hat g_\varepsilon''(-\eta)\hat g_\varepsilon(\eta)\hat g_\varepsilon(\xi)
+2\eta \xi \hat g_\varepsilon'(\xi)\hat g_\varepsilon(\eta) \hat g_\varepsilon'(-\eta)\Big)\ud\eta\Bigg) + \theta^2 O(\varepsilon^3)\Bigg]\ud\theta.
\end{split}
\ee
Since $\beta(\theta)$ is an even function, the first-order terms vanish.

By Theorem \ref{moments} there exists a constant $\lambda _T>0$ which do not depend on $\varepsilon$ such that
$$
\sup_{0<t<T} \left\{ \sum_{i=0}^4 \|\partial^i_{\xi^i}\hat g_\varepsilon (t,.)\|_{L^\infty}  \right\} \leq \lambda_T.
$$
Using equation \eqref{fourierdvpt}, we see that the family $(\hat g_\varepsilon)_\varepsilon$ is equicontinuous, so that we can use Ascoli's theorem,
which says that there exists a function $\hat g\in L^\infty((0,T)\times\R)$ such that, up to the extraction of a subsequence,
$$
\|\hat g_\varepsilon -\hat g\|_{L^\infty([t_1,t_2]\times \mathcal K)} \underset{\varepsilon \to 0}{\longrightarrow 0}
$$
for all $0<t_1<t_2<T$ and all compact set $\mathcal K \subset \R$. In addition, all the results of Theorem \ref{regularite} are still valid for
$\hat g$ on $(0,T)$. Therefore, thanks to both the uniform convergence for $\hat g$ and the convergence in $L^\infty((0,T)\times\R)$ weak-* for its
derivatives, we can pass to the limit in equation \eqref{fourierdvpt}, and we get, using classical formulae on the Fourier transform and the
conservation of mass and energy, that $\hat g$ satisfies the equation which is the Fourier transform of equation \eqref{FKBBE}.
\end{proof}

\begin{proof}[Proof of Theorem \ref{grazingformal}]
 To study the grazing collision limit, we needed some regularity on the Fourier transform of the solution. This is equivalent to have a uniform
bound on some higher moment of the solution. In proving Theorem \ref{moments}, we used the regularity of the mollified part, more precisely the
fact that this part was in $C^1$. In fact it could be enough to use the $H^{\frac{\nu}{2}}$ regularity which follows from the $H$-theorem. Indeed,
the terms that raise problems in the proof of Theorem \ref{moments} are $\tilde I_1$ and $\tilde I_2$. Let us see how to treat the first one. We
have
 \bee
\begin{split}
|\tilde I_1| &= \left|\int\beta(\theta)\sin\theta\cos^3\theta \int_{\R^2}vv_*(v^2-v_*^2)ff_*( f(v')- f(\tilde v')) \ud v\ud v_*\ud\theta \right| \\
&= \left| \int\beta(\theta)\sin\theta\cos^3\theta \int_{\R^2}vv_*(v^2-v_*^2)f(v)f(v_*) \Bigg[v(1-\cos\theta)\right]^\alpha \\
&\qquad\qquad\frac{f(v')-f(\tilde v')}{[v(1-\cos\theta)]^\alpha}\ud v\ud v_*\ud\theta \Bigg|\\
&\leq \int\beta(\theta)|\sin\theta||1-\cos\theta|^\alpha|\cos^3\theta| \left( \int_{\R^2}
\left(v^{1+\alpha}v_*(v^2-v_*^2)\right)^p f(v)^p f(v_*)^p \ud v\ud v_* \right)^{1/p}\cdot \\
&\qquad \cdot \left( \int_{\R^2} \frac{|f(v')-f(\tilde v')|^q}{|v(1-\cos\theta)|^{\alpha q}}\ud v\ud v_* \right)^{1/q} \ud\theta.
\end{split}
 \eee
In order to deal with moments not exceeding the fourth one, we use  H\"older's inequality, with $p$ such that $(3+\alpha)p =4$. Consequently $ q =
\frac{4}{1-\alpha}.$ Moreover, in order to recognize the semi-norm of $f$ in the Sobolev space $H^{\nu/2}$ in the last term of the product, we
need to set
$$
\alpha q = 1+\nu,
$$
and thus
$$
\alpha = \frac{1+\nu}{5+\nu}
$$
(note that $0<\alpha<1$). With these constants, we have
 \bee
\begin{split}
\left( \int_{\R^2} \frac{|f(v')-f(\tilde v')|^q}{|v(1-\cos\theta)|^{\alpha q}}\ud v\ud v_* \right)^{1/q}
&= \frac{1}{|1-\cos\theta|^{1/q}}\left( \int_{\R^2} \frac{|f(v')-f(\tilde v')|^q}
{|v(1-\cos\theta)|^{\alpha q}}\ud v\ud (1-\cos\theta)v_* \right)^{1/q}\\
&\leq \frac{2^{1-\frac{2}{q}}}{|1-\cos\theta|^{\frac{1}{q}}|\sin\theta|^{\frac{1}{q}}}
\|f(t)\|_{L^\infty}^{1-\frac{2}{q}}|f(t)|_{H^{\nu/2}}^{\frac{2}{q}}.
\end{split}
\eee
 Finally
\bee
\begin{split}
|\tilde I_1|&\leq C \int\beta(\theta) |\cos^3\theta||\sin\theta|^{1-\frac{1}{q}} |1-\cos\theta|^{\alpha-\frac{1}{q}} \ud\theta
\left(\int_\R v^4 f(t,v)\ud v\right)^{\frac{2}{p}}\|f(t)\|_{L^\infty}|f(t)|_{H^{\nu/2}}^{\frac{2}{q}}\\
%&\leq  C \int\beta(\theta) |\cos^3\theta||\sin\theta||1-\cos\theta|^{\alpha-\frac{1}{q}}
%\ud\theta\times \left(\int_\R v^4 f(t,v)\ud v\right)^{\frac{2}{p}}\\
%&\qquad\times\|f(t)\|_{L^\infty}|\log(1+\delta f(t))|_{H^{\nu/2}}^{\frac{2}{q}}(1+\delta \|f(t)\|_{L^\infty})^{\frac{2}{q}}.
\end{split}
\eee
 The integral in $\theta$ is finite if and only if
$$
2\left(\alpha - \frac{1}{q}\right)-\nu -\frac{1}{q}>-1.
$$
Recalling that
$$
q = \frac{4}{1-\alpha}
$$
and that
$$
\alpha = \frac{1+\nu}{5+\nu},
$$
this request is equivalent to
$$
\nu^2 +2\nu -4 <0,
$$
which is verified for $-1-\sqrt{5}<\nu<\sqrt{5}-1$. Combining this with the previous constraints on the parameter $\nu$, we obtain that our bound
of the fourth moment works for all $\nu$ verifying
$$
1<\nu< \sqrt{5}-1.
$$

The treatment of $\tilde I_2$ is quite more simple, since it requires only some changes of variable. Indeed, using the changes of variable $v
\mapsto v-v_*\sin\theta$ and $\theta \mapsto -\theta$
 \bee
\begin{split}
\tilde I_2 &= \int_{-\frac{\pi}{2}}^{\frac{\pi}{2}}\beta(\theta)\sin\theta\cos^3\theta
\int_{\R^2}vv_*(v^2-v_*^2)f(v) f(v_*) (f(\tilde v')-f(v))\ud v\ud v_*\ud\theta\\
&=\int_{-\frac{\pi}{2}}^{\frac{\pi}{2}}\beta(\theta)\sin\theta\cos^3\theta
\int_{\R^2}vv_*(v^2-v_*^2)f(v) f(v_*) f(\tilde v')\ud v\ud v_*\ud\theta\\
&=\int_{-\frac{\pi}{2}}^{\frac{\pi}{2}}\beta(\theta)\sin\theta\cos^3
\theta \cdot \\
& \null\quad \cdot\int_{\R^2}(v+v_*\sin\theta)v_*((v+v_*\sin\theta)^2-v_*^2)f(v+v_*\sin\theta) f(v_*) f(v)\ud v\ud v_*\ud\theta\\
&=-\int_{-\frac{\pi}{2}}^{\frac{\pi}{2}}\beta(\theta)\sin\theta\cos^3 \theta \cdot \\
& \null\quad \cdot\int_{\R^2}(v-v_*\sin\theta)v_*((v-v_*\sin\theta)^2-v_*^2)f(\tilde v') f(v_*) f(v)\ud v\ud v_*\ud\theta
\end{split}
\eee
 This implies
$$
\tilde I_2 \leq C(m,e) \left(\int_{-\frac{\pi}{2}}^{\frac{\pi}{2}}\beta(\theta)\sin^2\theta\ud\theta\right) \|f(t)\|_{L^\infty}\left(e^2 +
m^{\frac{1}{4}}\left(\int_\R v^4 f(t,v)\ud v\right)^{\frac{3}{4}} \right).
$$
\end{proof}

\section{Conclusions}

In this paper we investigated the asymptotic equivalence between the (mollified) Kac caricature of a Bose-Einstein gas and a nonlinear
Fokker-Planck type equation in the so-called grazing collision limit. The limit equation differs from the analogous one present in the literature
\cite{KQ93}, since in our case the linear diffusion has a diffusivity which depends on the solution itself, in order to guarantee the conservation
of energy. Our analysis refers to a mollified version of the equation, due to the difficulties of handle the third order nonlinearity present in
the Bose-Einstein correction. A further inside on the true model, done in the first part of the paper, shows that a proof of the boundedness of the
solution would be sufficient to avoid the presence of the mollifier.

\paragraph{Acknowledgement:} Support from the Italian Minister for Research, project ``Kinetic and
hydrodynamic equations of complex collisional systems'' is kindly
acknowledged. Thibaut Allemand thanks the Department of Mathematics
of the University of Pavia, where a part of this research has been
carried out, for the kind hospitality.

\bibliographystyle{acm}
%\bibliography{ArticleBoseEinstein3}

\end{document}